\documentclass{article} 
\usepackage{graphicx}

\usepackage{physics}

\usepackage{amsmath}
\usepackage{amssymb}
\usepackage{amsfonts}

\usepackage{mathtools}

\usepackage{amsthm}

\usepackage{mathrsfs}

\usepackage{mleftright}

\usepackage{bussproofs}
\usepackage{tikz}\usetikzlibrary{cd}
\usetikzlibrary{arrows.meta}

\usepackage{enumitem}

\usepackage{comment}

\mleftright

\newcommand{\defeq}{\mathrel{\mathop{:}}=}

\newcommand{\complex}{\mathbf{C}}

\newcommand{\projsp}{\mathbf{P}}

\DeclareMathOperator{\id}{id}

\DeclareMathOperator{\Aut}{Aut}
\DeclareMathOperator{\GL}{GL}

\DeclareMathOperator{\PGL}{PGL}

\theoremstyle{definition}
\newtheorem{dfn}{Definition}[section]

\theoremstyle{plain}
\newtheorem{proposition}[dfn]{Proposition}
\newtheorem{lemma}[dfn]{Lemma}
\newtheorem{theorem}[dfn]{Theorem}

\theoremstyle{remark}
\newtheorem{remark}[dfn]{Remark}

\title{On the $K3$ surface with $\mathfrak{S}_4 \times \mathfrak{S}_4$ action}
\author{HAYATO NUKUI}
\date{February 21, 2026}
\begin{document}

\maketitle

\begin{abstract}
    By a lattice theoretic approach, Brandhorst--Hashimoto has made the list of K3 surfaces with finite groups of automorphisms which properly contain a maximal symplectic automorphism group. We give $3$ different explicit descriptions to the $K3$ surface with an action of $\mathfrak{S}_4\times \mathfrak{S}_4$, with various characterizations, and construct an explicit isomorphism to the Schur's quartic. We also calculate the intersection of the two polarization-preserving finite automorphism groups.
\end{abstract}

\section{Introduction}

An algebraic surface $S$ over an algebraically closed field $k$ is called $K3$ \textit{surface} if its canonical class is trivial and $\dim H^1(S,\mathcal{O}_S)=0$. 
An automorphism of a $K3$ surface $S$ is called \textit{symplectic} if it acts trivially on $H^0(S,\omega_S)$. Nikulin \cite{nikulin1979finite} classifies the symplectic action of finite abelian group on complex $K3$ surface and Mukai \cite{mukai1988finite} shows that in general a finite symplectic automorphism group of a complex $K3$ surface is a subgroup of one of $11$ maximal groups.
Brandhorst and Hashimoto \cite{Brandhorst2021ExtensionsOM} classifies complex $K3$ surfaces with finite groups of automorphisms which are proper extension of the maximal symplectic finite automorphism groups. The classification is lattice-theoretical and some projective models of such $K3$ surfaces are not determined yet.

In this paper we construct one with a faithful action of $\mathfrak{S}_4\times\mathfrak{S}_4$ of invariant polarization degree $24$ (78c in classification of \cite{Brandhorst2021ExtensionsOM}), which is known to be unique. 

Brandhorst and Hashimoto determine its transcendental lattice and it is isomorphic to those of known $K3$ surfaces with remarkable properties. Shioda--Inose correspondence immediately implies that those surfaces are isomorphism abstractly over $\complex$. We also construct concrete expressions of isomorphisms between them in positive characteristics.
\begin{theorem}[main theorem]\label{promain1}
    \begin{enumerate}
        \item The polarized complex $K3$ surface with a faithful action of $\mathfrak{S}_4\times\mathfrak{S}_4$ admits a projective model in $\projsp^1\times\projsp^1\times\projsp^1\times\projsp^1$ given by
    \[
        S=\left\{
        \begin{array}{l}
            s tuv+1=0 \\
            st+uv+su+tv+sv+tu=0,
        \end{array}
        \right.
    \]
    where $s$, $t$, $u$ and $v$ are the inhomogeneous coordinates of $\projsp^1\times\projsp^1\times\projsp^1\times\projsp^1$.

    The action of $\mathfrak{S}_4\times\mathfrak{S}_4$ is generated by the permutations of the coordinates and the diagonal actions of 
    \[
        \alpha=\begin{pmatrix}
            \zeta_8&0 \\ 0&\zeta_8^7
        \end{pmatrix},
        \;\beta=\sqrt{\frac{-1}{2}}
        \begin{pmatrix}
            1&1\\1&-1
        \end{pmatrix}\in \PGL_2(\complex).
    \]
    \item Suppose $\operatorname{ch} k\neq 2,3$. $S$ is isomorphic to:
    \begin{itemize}
        \item the quartic surface with a projective action of the group of \textrm{GAP} Id $[1152, 157515]$
            \[
                S_{T_{192}}\defeq (x^4+y^4+z^4+w^4-2\sqrt{-3}(x^2y^2+z^2w^2)=0)\subset\projsp^3, 
            \]
        \item the Schur's quartic
            \[
                S_{Sch}\defeq (x^4-xy^3=z^4-zw^3)\subset\projsp^3
            \]
        \item the Kummer surface associated with the direct product of two elliptic curves
            \[
                E_1=\left(y_1^2=x(x_1-1)\left(x_1-\frac{1+\sqrt{-3}}{2}\right)\right)
            \]
            and
            \[
                E_2=\left(y_1^2=(x_1^2-1)\left(x_1^2-\frac{1+\sqrt{-3}}{2}\right)\right).
            \]
    \end{itemize}
    \label{promain12}
    \end{enumerate}
\end{theorem}

Using this isomorphism between $S$ and $S_{Sch}$, we determine the intersection of two polarization-preserving automorphism groups in the automorphism group.

\begin{theorem}\label{promain2}
    Let $h_{24}$ and $h_4$ be the class of the polarization of $S$ and $S_{Sch}$, respectively. The intersection of $\Aut(S,h_{24})$ and $\Aut(S,h_4)$ in $\Aut(S)$ is isomorphic to $C_4\rtimes \mathfrak{S}_4$ and its symplectic subgroup is isomorphic to $C_2\times\mathfrak{S}_4$.
\end{theorem}
Here, $\Aut(S,h)$ is the subgroup of $\Aut(S)$ which fixes the polarization class $h$.

\subsubsection*{Contents of the paper}In section 2 we recall basic facts of $K3$ surfaces and introduce notations from group theory. Section 3 contains construction of the surface, proof of the main theorem and various characterizations of this surface. In section 4 we calculate the intersection of the two polarization-preserving automorphism groups.

\subsubsection*{Notations}

$\mathfrak{S}_n$ and $\mathfrak{A}_n$ denote the symmetric group and the alternating group of degree $n$, respectively.. $Q_8$ is the quaternion group of order $8$ and $C_n$ is the cyclic group of order $n$. 

We assume that the characteristics of $k$ is not equal to $2$ or $3$ unless otherwise stated.

We may display a lattice in terms of its Gram matrix.

\subsubsection*{Acknowledgement} 

The author thanks Hisanori Ohashi for suggesting this problem, Yuya Matsumoto and Kenji Hashimoto for comments, and his supervisor Yuji Odaka for comments and encouragements.

\section{Preliminaries}

\subsection{Binary octahedral group and binary tetrahedral group}

We introduce some known facts from group theory. Let $\zeta_m$ denote a primitive $m$th root of unity and let $\omega$ be a primitive cube root of unity.

\begin{dfn}
    Suppose that $\operatorname{ch} k\neq 2$. 
    \textit{Binary octahedral group} $BO$ is the subgroup of $\GL_2(k)$ generated by two elements
\[
    \alpha=\begin{pmatrix}
        \zeta_8&0 \\ 0&\zeta_8^7
    \end{pmatrix},
    \;\beta=\sqrt{\frac{-1}{2}}
    \begin{pmatrix}
        1&1\\1&-1
    \end{pmatrix}.
\]
The quotient by its center 
\[
    \left\{\begin{pmatrix}
        1&0\\0&1
    \end{pmatrix},\begin{pmatrix}
        -1&0\\0&-1
    \end{pmatrix}\right\}
\]
is called \textit{octahedral group} $O$. $O$ is isomorphic to the symmetric group $\mathfrak{S}_4$ of degree $4$.
\end{dfn}

By calculation by GAP\cite{GAP4}, we have
\begin{proposition}\label{character}
    The character table of $BO$ is:
    \[
    \begin{array}{c|rrrrrrrr}
  \rm class&\rm1&\rm2&\rm3&\rm4A&\rm4B&\rm6&\rm8A&\rm8B\cr
  \rm size&1&1&8&6&12&8&6&6\cr
\hline
  \rho_{1}&1&1&1&1&1&1&1&1\cr
  \rho_{2}&1&1&1&1&-1&1&-1&-1\cr
  \rho_{3}&2&2&-1&2&0&-1&0&0\cr
  \rho_{4}&2&-2&-1&0&0&1&\sqrt{2}&-\sqrt{2}\cr
  \rho_{5}&2&-2&-1&0&0&1&-\sqrt{2}&\sqrt{2}\cr
  \rho_{6}&3&3&0&-1&-1&0&1&1\cr
  \rho_{7}&3&3&0&-1&1&0&-1&-1\cr
  \rho_{8}&4&-4&1&0&0&-1&0&0\cr
\end{array}
\]
$\alpha$ is in the conjugacy class $4B$ and $\beta$ is in $8A$.
\end{proposition}

\subsection{Preliminaries on $K3$ surfaces}

$K3$ surfaces on the list of Brandhorst and Hashimoto are singular, that is, their transcendental lattices are of rank $2$. The theorem below is a well-known result on singular $K3$ surfaces.

\begin{theorem}[{\cite[Theorem 4]{shioda1977singular}}]\label{shiodainose}
    There is a bijection between the set of singular $K3$ surfaces up to isomorphisms and the set of positive definite, even, oriented lattice of rank $2$ up to isomorphism which associates with a singular $K3$ surface its transcendental lattice.
    \begin{align*}
        \{\textrm{singular } K3 \textrm{ surfaces}\}/{\sim} \xrightarrow{1:1} \{\textrm{positive definite even oriented lattice of rank } 2\}/{\sim}
    \end{align*}
\end{theorem}

Those $K3$ surfaces come with elliptic fibrations. This is a criterion for a divisor to be a fiber of an elliptic fibration.
\begin{theorem}[{\cite[Proposition 3.10]{Huybrechts_2016}}]\label{ellfiber}
    Assume characteristics of $k$ is not equal to $2$ or $3$ and let $D$ be an effective divisor on a $K3$ surface with $D^2=0$. If $D$ is nef and primitive, then there is a smooth elliptic curve such that $E\in |D|$.
\end{theorem}

Mukai's classification of finite symplectic automorphism groups also applies to so-called tame case in positive characteristics. 
\begin{theorem}[{\cite{DolgKeum},\cite{mukai1988finite}}]\label{Mukai}
    Let $G$ be a finite group of symplectic automorphisms of a $K3$ surface. If characteristic $p=0$ or if $p$ is coprime to the order of $G$, $G$ is isomorphic to a subgroup of the Mathieu group $M_{23}$ which has at least $5$ orbits in its natural permutation action on the set of $24$ elements, that is, a subgroup of one of $11$ groups
    \[
        T_{48}, N_{72}, M_9, \mathfrak{S}_5, L_2(7), H_{192}, T_{192}, \mathfrak{A}_{4,4}, \mathfrak{A}_6, F_{384}, \text{and}\, M_{20}
    \]
\end{theorem}

Later we will use groups $\mathfrak{A}_{4,4}$ and $T_{192}$ from the Mukai's classification.

The group $\mathfrak{A}_{4,4}$ is the intersection of $\mathfrak{A}_8$ and the image of the natural embedding $\mathfrak{S}_4\times \mathfrak{S}_4$ into $\mathfrak{S}_8$. 
    
Let $F$ be the central product $Q_8\ast Q_8$. For two automorphisms $\alpha_1$, $\alpha_2$ of $Q_8$, we denote by $\alpha_1\ast\alpha_2$ the automorphism $F\ni x\ast y\mapsto \alpha_1(x)\ast\alpha_2(y)$. $T_{192}$ is the semi-direct product $(Q_8\ast Q_8)\rtimes\langle\alpha^{-1}\ast\alpha,\tau\rangle$, where $\alpha$ is an automorphism of order $3$ and $\tau$ is the factor-change involution. $T_{192}$ is isomorphic to the centralizer of a permutation of type $(2)^4$ in $\mathfrak{A}_8$.

We recall some $K3$ surfaces with a prominent property. 

The quartic surface $S_{T_{192}}$ defined by
\[
    x^4+y^4+z^4+w^4-2\sqrt{-3}(x^2y^2+z^2w^2)=0
\]
admits an action by $T_{192}$ \cite{mukai1988finite}. The action of $Q_8\ast Q_8$ is generated by
\[
    \begin{pmatrix}
        I&0 \\
        0&E
    \end{pmatrix},
    \begin{pmatrix}
        J&0 \\
        0&E
    \end{pmatrix},
    \begin{pmatrix}
        E&0 \\
        0&I
    \end{pmatrix},
    \begin{pmatrix}
        E&0 \\
        0&J
    \end{pmatrix}\in\PGL_4(k),
\]
where
\[
    E=\begin{pmatrix}
        1&0\\0&1
    \end{pmatrix},
    I=\begin{pmatrix}
        0&1\\-1&0
    \end{pmatrix},
    J=\begin{pmatrix}
        \sqrt{-1}&0\\0&-\sqrt{-1}
    \end{pmatrix}.
\]
The action by $\alpha^{-1}\ast\alpha$ is
\[
    \frac{1+\sqrt{-1}}{2}
    \begin{pmatrix}
        \omega&\omega&0&0\\
        -\zeta_{12}&\zeta_{12}&0&0\\
        0&0&\zeta_{12}^5&-\omega^2\\
        0&0&\zeta_{12}^5&\omega^2
    \end{pmatrix}
\]
and the action by $\tau$ is
\[
    \begin{pmatrix}
        0&E \\
        E&0
    \end{pmatrix}.
\]
$S_{T_{192}}$ also admits a non-symplectic automorphism
\[
    \begin{pmatrix}
        \frac{-1+\sqrt{-1}}{2}
        \begin{pmatrix}
            \omega&\omega \\
            -\zeta_{12}&\zeta_{12}
        \end{pmatrix}
        &0\\
        0&I
    \end{pmatrix}
\]
of order $6$, which generates a group of order $1152$ together with the above $T_{192}$.

The Schur's quartic is a quartic surface defined by $x^4-xy^3=z^4-zw^3$, named after F. Schur, who proved this surface contains $64$ lines. This surface is known to be the unique quartic to have $64$ lines, which is the maximum number of lines on a quartic surface \cite{DegItenSer}. $16$ lines out of $64$ are in form of $(f=g=0)$ where $f$ and $g$ are the linear factors of each side of the defining equation. The other $48$ lines are in the orbit of the line $(x-z=y-w=0)$ under the action of projective automorphisms.

\section{Main theorem and proof}

Let $X$ denote the direct product of $4$ projective lines $(\projsp^1)^4$ and let $H$ denote the line bundle $\mathscr{O}_{X}(1,1,1,1)$. Let $\tilde{G}\defeq BO\times \mathfrak{S}_4$ and $G\defeq O\times\mathfrak{S}_4$.

$\GL_2(k)\times \mathfrak{S}_4$ acts on $H^0(X,H)$ as follows: the first factor $\GL_2(k)$ acts diagonally and the second factor $\mathfrak{S}_4$ acts as permutation of coordinates. The maps
\begin{align*}
    BO\ni \varphi\quad &\longmapsto \quad(\varphi,\id)\in\GL_2(k)\times \mathfrak{S}_4, \\
    \mathfrak{S}_4\ni\sigma\quad&\longmapsto \quad\left(
    \begin{pmatrix}
        1&0\\0&1
    \end{pmatrix},\sigma\right)\in\GL_2(k)\times \mathfrak{S}_4
\end{align*}
are embeddings and their images commute with each other. Hence we have embedding $\tilde{G}\hookrightarrow\GL_2(k)\times \mathfrak{S}_4$ and action of $\tilde{G}$ on $H^0(X,H)$ by restriction. We define the action of $G$ on $X$ as the induced action by it.

\begin{lemma}\label{lemma24}
    An intersection of two hypersurface of type $(1,1,1,1)$ in $X$ is a $K3$ surfaces with polarization degree $2d=24$ if it is non-singular.
\end{lemma}
\begin{proof}
    Suppose that $S$ is a complete intersection of two hypersurface $Y$, $Z$ of type $(1,1,1,1)$.
    
    By adjunction formula, the canonical sheaf of $S$ is trivial.

    By short exact sequences
    \[
    \begin{tikzcd}
        0\ar[r]&\mathscr{O}_X(-Y)\ar[r]&\mathscr{O}_X\ar[r]&\mathscr{O}_Y\ar[r]&0, \\
        0\ar[r]&\mathscr{O}_X(-2Y)\ar[r]&\mathscr{O}_X(-Y)\ar[r]&\mathscr{O}_Y(-Y)\ar[r]&0,
    \end{tikzcd}
    \]
    we have $H^1(Y,\mathscr{O}_Y)=0$ and $H^2(Y,\mathscr{O}_X(-Y)|_Y)=0$. By these and the short exact sequence
    \[
    \begin{tikzcd}
        0\ar[r]&\mathscr{O}_X(-Y)|_Y\ar[r]&\mathscr{O}_Y\ar[r]&\mathscr{O}_S\ar[r]&0,
    \end{tikzcd}
    \]
    we have $H^1(S,\mathscr{O}_S)=0$.
\end{proof}

Therefore, to find $G$-stable $K3$ surfaces in $X$, one must determine $2$-dimensional subrepresentation of the action of $\tilde{G}$ on $H^0(X,H)$.

Let $\sigma_i$ denote the elementary symmetric polynomial of degree $i$ in variable $s$, $t$, $u$ and $v$ where $s$, $t$, $u$ and $v$ are the inhomogeneous coordinates of each factor of $X$.
\begin{proposition}\label{stables}
    $2$-dimensional subrepresentation of $H^0(X,H)$ is either of the two below:
    \begin{enumerate}
        \item $\langle st+uv+\omega(su+tv)+\omega^2(sv+tu),st+uv+\omega^2(su+tv)+\omega(sv+tu)\rangle$\;\text{and} 
        \label{enuone} \\
        \item $\langle \sigma_4+1, \sigma_2\rangle$,
    \end{enumerate}
    where $\omega$ is a primitive cube root of unity.
\end{proposition}
\begin{proof}
    As the action by the second factor of $\tilde{G}$ on $H^0(X,H)$ does not change the degree of monomials, $H^0(X,H)$ decomposes into $4$ subrepresentations of $\mathfrak{S}_4$ as follows:
    \[
        H^0(X,H)=\langle s t u v\rangle\oplus\langle stu,stv,suv,tuv\rangle\oplus\langle st,su,sv,tu,tv,uv\rangle\oplus\langle s,t,u,v\rangle\oplus\langle 1\rangle.
    \]

    Obviously $\langle stuv\rangle$ and $\langle 1\rangle$ are irreducible, while $\langle stu,stv,suv,tuv\rangle$ and $\langle s,t,u,v\rangle$ are sum of the trivial representation and a $3$-dimensional irreducible representation, respectively. As for the last component $\langle st,su,sv,tu,tv,uv\rangle$ we have 
    \[
        \langle st,su,sv,tu,tv,uv\rangle 
        =\langle st+su+sv+tu+tv+uv\rangle 
        \oplus\langle f_1,f_2\rangle 
        \oplus\langle st-uv,su-tv,sv-tv\rangle.
    \]
    Here, 
    \begin{align*}
        f_1&=(st+uv)+\omega(su+tv)+\omega^2(sv+tu), \\
        f_2&=(st+uv)+\omega^2(su+tv)+\omega(sv+tu)
    \end{align*}
    and summands on the right hand side are irreducible representations.

    The action of second factor $\mathfrak{S}_4$ of $\tilde{G}$ on a $2$-dimensional subrepresentation of $\tilde{G}$ is either irreducible, namely same as $\langle f_1, f_2\rangle$, or sum of the trivial representations and the sign representations.

    Let $V$ be a $2$-dimensional subrepresentation of $\tilde{G}$. Then $V$ is, as a representation of the second factor $\mathfrak{S}_4$ of $\tilde{G}$, either irreducible (hence equal to $\langle f_1, f_2\rangle$) or is the sum of two $1$-dimensional representations.

    Let us consider the first case.
    
    Since actions by the generators on $V$ are 
    \begin{align*}
        \alpha f_1
        &=f_1, \\
        \alpha f_2
        &=f_2, \\
        \beta f_1
        &=f_1, \\
        \beta f_2
        &=f_2,
    \end{align*}
    we conclude that $V$ is a subrepresentation.

    In the latter case, we have to find $2$-dimensional subrepresentations of $W:=\langle\sigma_4,\sigma_3,\sigma_2,\sigma_1,1\rangle$ under the action of the first factor $BO$ of $\tilde{G}$. The generators act on $W$ as follows.
    \begin{alignat*}{2}
        \alpha\sigma_4 &=-\sigma_4,\qquad & \beta\sigma_4&=\frac{1}{4}(\sigma_4+\sigma_3+\sigma_2+\sigma_1+1) \\
        \alpha\sigma_3 &=i\sigma_3,\qquad & \beta\sigma_3&=\sigma_4+\frac{1}{2}\sigma_3-\frac{1}{2}\sigma_1-1 \\
        \alpha\sigma_2 &=\sigma_2,\qquad & \beta\sigma_2&=\frac{3}{2}\sigma_4-\frac{1}{2}\sigma_2+\frac{3}{2}\sigma_1 \\
        \alpha\sigma_1 &=-i\sigma_1,\qquad & \beta\sigma_1&=\sigma_4-\frac{1}{2}\sigma_3+\frac{1}{2}\sigma_1-1 \\
        \alpha 1 &=-1,\qquad & \beta 1&=\frac{1}{4}(\sigma_4-\sigma_3+\sigma_2-\sigma_1+1).
    \end{alignat*}
    Let $\chi$ be the character of this representation $W$. From the calculation above, it follows that $\chi(\alpha)=-1$ and $\chi(\beta)=1$.
    
    By character theory, we deduce from Proposition~\ref{character} that the representation of $BO$ on $W$ is sum of two irreducible representations $\rho_3$ and $\rho_6$, and we have decomposition
    \[
        W\cong\langle\sigma_4+1,\sigma_2\rangle\oplus\langle\sigma_4-1,\sigma_3,\sigma_1\rangle.
    \]
    Therefore $\langle\sigma_4+1,\sigma_2\rangle$ is a representation of $\tilde{G}$ and there is no other $2$-dimensional subrepresentation in $H^0(X,H)$.
\end{proof}

The surface defined by first subrepresentation in Proposition~\ref{stables} is singular.

\begin{proposition}
    $T\defeq((st+uv)+\omega(su+tv)+\omega^2(sv+tu)=(st+uv)+\omega^2(su+tv)+\omega(sv+tu)=0)$ is singular at $(0,0,0,0)$.
\end{proposition}
\begin{proof}
    Since both of defining equations are homogeneous polynomial of degree $2$, $T$ is singular at $(0,0,0,0)$ by Jacobian criterion.
\end{proof}

\begin{remark}
    The Singular locus of $T$ is the curve $(s=t=u=v)$.
\end{remark}

The other subrepresentation does yield a non-singular surface.
\begin{theorem}[main theorem]\label{mainthm}
    Let $S$ be the surface $(\sigma_4+1=\sigma_2=0)$.
    \begin{enumerate}
        \item In characteristic $p\neq 2,3$, $S$ is non-singular and is a $K3$ surface.  \label{main1}
        \item In characteristic $p\neq 2,3$, $\mathfrak{S}_4\times\mathfrak{S}_4$ acts faithfully on $S$ with invariant polarization degree $24$ and the symplectic subgroup is isomorphic to $\mathfrak{A}_{4,4}$. \label{main2}
        \item For a $K3$ surface $T$ over the complex field, following conditions are equivalent to each other:
            \begin{enumerate}
                \item $T$ is isomorphic to $S$.\label{main31}
                \item $\mathfrak{S}_4\times\mathfrak{S}_4$ acts faithfully on $T$ with invariant polarization degree $24$ and the symplectic subgroup is isomorphic to $\mathfrak{A}_{4,4}$.\label{main32}
                \item The transcendental lattice of $T$ is
                \[\begin{pmatrix}
                    8&4\\4&8
                \end{pmatrix}.\]\label{main33}
            \end{enumerate}
        \label{main3}
    \end{enumerate}
\end{theorem}
\begin{proof}
    (\ref{main1})
    Again by symmetry of the defining equations, we only have to consider the case where the point is of the form $([a:1],[b:1],[c:1],[d:1])$, $([a:1],[b:1],[c:1],[1:0])$, or $([a:1],[b:1],[1:0],[1:0])$ to determine if a point is a singular point.

    If $([a:1],[b:1],[c:1],[d:1])\in (\projsp^1)^4$ is contained in $S$, $a, b, c, d\neq 0$ holds since $abcd+1=0$. 
    
    Let $f=\sigma_4+1$ and $g=\sigma_2$. By Jacobian criterion, singular points $([a:1],[b:1],[c:1],[d:1])\in S$ satisfies
    \begin{align}
        \pdv{f}{s}\pdv{g}{t}-\pdv{f}{t}\pdv{g}{s}=(b-a)(c+d)cd&=0, \label{mainone}\\
        \pdv{f}{t}\pdv{g}{u}-\pdv{f}{u}\pdv{g}{t}=(c-b)(a+d)ad&=0, \label{maintwo}\\
        \pdv{f}{u}\pdv{g}{v}-\pdv{f}{v}\pdv{g}{u}=(d-c)(a+b)ab&=0. \label{mainthree}
    \end{align}
    When $b=a$, we have $c=b$ or $a+d=0$ by (\ref{maintwo}). 
    
    It follows that $c=d$ by (\ref{mainthree}) and $a=\pm c$ by (\ref{maintwo}). Hence the equation
    \begin{align*}
        f(a,a,c,c)=a^4+1=0, \\
        g(a,a,c,c)=2a^2\pm 4a^2=0
    \end{align*}
    holds (the double sign corresponds to $a=\pm c$). This has no solutions.

    Singular points $([a:1],[b:1],[c:1],[1:0])\in S$ satisfies
    \[
        abc=a+b+c=0
    \]
    and
    \[
        \rank
        \begin{pmatrix}
            bc & ac & ab & 1 \\
            1 & 1 & 1 & ab+bc+ca \\
        \end{pmatrix}
        \leq 1.
    \]
    Thus $ab=bc=ca$ and $3a^2b^2=3b^2c^2=3c^2a^2=1$. But this contradicts $abc=0$. 

    Points of the form $([a:1],[b:1],[1:0],[1:0])$ are not on $S$.

    Therefore $S$ is non-singular unless characteristic $p\neq 2,3$. $S$ is a $K3$ surface by Lemma~\ref{lemma24}.

    (\ref{main2})
    $S$ admits a faithful action of $G=\mathfrak{S}_4\times O\cong\mathfrak{S}_4\times\mathfrak{S}_4$ with invariant polarization degree $24$ by Proposition~\ref{stables}. 

    Let $G_s$ be the symplectic subgroup of $G$. Since $G_s$ acts trivially on the $1$-dimensional vector space $H^0(S,\omega_S)$ by definition, $G/G_s$ is isomorphic to finite subgroup of $k^{\times}$. Therefore it is a cyclic group. Among normal subgroups of $\mathfrak{S}_4\times\mathfrak{S}_4$, four groups
    \[
        \mathfrak{S}_4\times\mathfrak{S}_4, \mathfrak{S}_4\times\mathfrak{A}_4, \mathfrak{A}_4\times\mathfrak{S}_4\;\text{and}\; 
        \mathfrak{A}_{4,4}
    \]
    satisfy the condition that the quotient is cyclic. Since the orders of these four groups are coprime to $p$ and the groups other than $\mathfrak{A}_{4,4}$ are not isomorphic to subgroups of the maximum groups in Mukai's list, we have $G_s\cong\mathfrak{A}_{4,4}$ by Theorem~\ref{Mukai}.

    (\ref{main3})
    By Theorem~\ref{shiodainose}, we only have to show the implications $(a)\Rightarrow(b)\Rightarrow(c)$ for (\ref{main3}). $(a)\Rightarrow(b)$ is due to (\ref{main2}).

    Suppose that $T$ satisfies $(b)$. Then by \cite[Proposition 5.10.]{Brandhorst2021ExtensionsOM}, $T$ is isomorphic to 78c of \cite{Brandhorst2021ExtensionsOM} and its transcendental lattice is
    \[
    \begin{pmatrix}
        8&4\\4&8
    \end{pmatrix}.
    \]
\end{proof}

\begin{proposition}\label{mainthm2} 
    In characteristics $p\neq 2,3$, $S$ is isomorphic to:
    \begin{itemize}
        \item 
            \[
                S_{T_{192}}= (x^4+y^4+z^4+w^4-2\sqrt{-3}(x^2y^2+z^2w^2)=0)\subset\projsp^3,
            \]
        \item the Schur's quartic
            \[
                S_{Sch}= (x^4-xy^3=z^4-zw^3)\subset\projsp^3,
            \]
        \item the Kummer surface associated with the direct product of two elliptic curves
            \[
                E_1=\left(y_1^2=x(x_1-1)\left(x_1-\zeta_6\right)\right)
            \]
            and
            \[
                E_2=\left(y_1^2=(x_1^2-1)\left(x_1^2-\zeta_6\right)\right).
            \]
            which are isomorphic to $\mathbf{C}/\mathbf{Z}\oplus\mathbf{Z}\omega$ and $\mathbf{C}/\mathbf{Z}\oplus\mathbf{Z}\sqrt{-3}$ over $\complex$, respectively.
    \end{itemize}
\end{proposition}
\begin{proof}
    In the case $k=\complex$, we utilize Theorem~\ref{shiodainose}. The transcendental lattice of $S$ and $S_{T_{192}}$ are isomorphic, being
    \[\begin{pmatrix}
        8&4\\4&8
    \end{pmatrix}\]
    by \cite{Brandhorst2021ExtensionsOM}. The transcendental lattice of the product of $E_\omega\defeq\mathbf{C}/\mathbf{Z}\oplus\mathbf{Z}\omega$ and $E_{\sqrt{-3}}\defeq\mathbf{C}/\mathbf{Z}\oplus\mathbf{Z}\sqrt{-3}$ can be calculated after \cite[Section 3]{10.1007/BFb0066163}, and is
    \[\begin{pmatrix}
        4&2\\2&4
    \end{pmatrix}.\]
    Thus, the transcendental lattice of the Kummer surface $\operatorname{Km}(E_\omega\times E_{\sqrt{-3}})$ is isomorphic to those of $S$ and $S_{T_{192}}$.

    $j$-invariants of the elliptic curves $E_{\omega}$ and $E_{\sqrt{-3}}$ are $0$ and $2\cdot 30^3$ \cite[Chapter 3, Section 12]{Coxprime}, respectively. Therefore the elliptic curves are isomorphic to $E\defeq(y^2=4x^3-1)$ and $E'\defeq(y^2=4x^3-15x-11)$. $E$ is isomorphic to double covering of $\projsp^1$ ramified over the four points $1$, $\omega$, $\omega^2$ and $\infty$ and $E'$ is isomorphic to double covering of $\projsp^1$ ramified over the four points $-1$, $\frac{1+2\sqrt{3}}{2}$, $\frac{1-2\sqrt{3}}{2}$ and $\infty$. Since their cross ratios coincide with those of $(0,1,\zeta_6, \infty)$ and $(\pm 1,\pm\zeta_{12})$, it follows that $E\cong E_1\defeq (y^2=x(x-1)(x-\zeta_6))$ and $E'\cong E_2\defeq (y^2=(x^2-1)(x^2-\zeta_6))$.

    Let $\tau'\in \Aut(E_2)$ denote the involution $(x,y)\mapsto (-x,-y)$. $\tau'$ has no fixed point. The quotient of $E_2$ by $\tau'$ is isomorphic to $E_1$ by the morphism $f:E_2/\langle\tau\rangle\to E_1, \overline{(x,y)} \mapsto (x^2,xy)$. 
    Let $g$ be the morphism
    \begin{align*}
        E_2\times E_2&\longrightarrow E_1\times E_2, \\
        (p,q)&\longmapsto (f(\overline{p}), p+q).
    \end{align*}
    Here, additive identity is arbitrary. As $g\circ(\tau'\times\tau')=g$, $g$ induces the isomorphism $(E_2\times E_2)/\langle\tau'\times\tau'\rangle\to E_1\times E_2$. Isomorphism between $S_{T_{192}}$ and the Kummer surface $\operatorname{Km}((E_2\times E_2)/\langle\tau'\times\tau'\rangle)$ is given in \cite{MizukamiKummer}.

    To give an isomorphism $S_{T_{192}}\simeq S_{Sch}$, define a matrix
    \[
        M\defeq
        \begin{pmatrix}
            \zeta_{12}-\zeta_6&-1 \\
            1-\sqrt{-1}&(1+\sqrt{-1})(\zeta_{12}-\zeta_6)
        \end{pmatrix}.
    \]
    $M$ maps the four roots of $x^4-2\sqrt{-3}x^2+1=0$ to the four roots of $x^4-x=0$ as a projective transformation. Therefore isomorphism $S_{T_{192}}\simeq S_{Sch}$ is given by
    \[
        \begin{pmatrix}
            M&0 \\ 0&\zeta_8M
        \end{pmatrix}\in\PGL_4(k).
    \]

    Next we construct an isomorphism between $S$ and $S_{Sch}$. Consider the rational map $\varphi:\projsp^3\dashrightarrow (\projsp^1)^4$ defined by
    \[
        [x:y:z:w]\;\longmapsto\;([x:-z],[x-y:z-w],[x-\omega y:z-\omega^2 w],[x-\omega^2y:z-\omega w]).
    \]

    The image of $\varphi$ is contained in the threefold $(\sigma_2=0)\subset (\projsp^1)^4$. Indeed, we have
    \begin{align*}
        &x(x-y)(z-\omega^2w)(z-\omega w)-z(z-w)(x-\omega^2y)(x-\omega y) \\
        &+x(x-\omega y)(z-w)(z-\omega w)-z(z-\omega^2 w)(x-y)(x-\omega^2 y) \\
        &+x(x-\omega^2 y)(z-w)(z-\omega^2 w)-z(z-\omega w)(x-y)(x-\omega y) \\
        &=(x^2z^2+x^2zw+x^2w^2-xyz^2-x yzw-xyw^2) \\
        &-(x^2z^2+xyz^2+y^2z^2-x^2zw-x yzw-y^2zw) \\
        &+(x^2z^2+\omega^2x^2zw+\omega x^2w^2-\omega xyz^2-x yzw-\omega^2xyw^2) \\
        &-(x^2z^2+\omega xyz^2+\omega^2y^2z^2-\omega^2x^2zw-x yzw-\omega y^2zw) \\
        &+(x^2z^2+\omega x^2zw+\omega^2x^2w^2-\omega^2xyz^2-x yzw-\omega xyw^2) \\
        &-(x^2z^2+\omega^2xyz^2+\omega y^2z^2-\omega x^2zw-x yzw-\omega^2y^2zw) \\
        &=0.
    \end{align*}
    Furthermore, the restriction of $\varphi$ on $S_{Sch}$ maps points on $S_{Sch}$ into $S$. We also denote the restriction by $\varphi$.
    
    If a point $P=([s_0:s_1],[t_0:t_1],[u_0:u_1],[v_0:v_1])\in S$ is in the image of $\varphi$, i.e. there is a point $[x:y:z:w]\in S_{Sch}$ such that $\varphi([x:y:z:w])=P$, they satisfy the equations
    \begin{align*}
        s_1x&=-s_0z, \\
        t_1(x-y)&=t_0(z-w), \\
        u_1(x-\omega y)&=u_0(z-\omega^2w), \\
        v_1(x-\omega^2 y)&=u_0(z-\omega w).
    \end{align*}
    and hence we have
    \begin{align*}
        0
        &=(x-y)+\omega(x-\omega y)+\omega^2(x-\omega^2y) \\
        &=\frac{t_0}{t_1}(z-w)+\omega\frac{u_0}{u_1}(z-\omega^2w)+\omega^2\frac{v_0}{v_1}(z-\omega w) \\
        &=\left(\frac{t_0}{t_1}+\omega\frac{u_0}{u_1}+\omega^2\frac{v_0}{v_1}\right)z
        -\left(\frac{t_0}{t_1}+\frac{u_0}{u_1}+\frac{v_0}{v_1}\right)w, \\
        0
        &=(z-w)+\omega^2(z-\omega^2w)+\omega(z-\omega w) \\
        &=\frac{t_1}{t_0}(x-y)+\omega^2\frac{u_1}{u_0}(x-\omega y)+\omega\frac{v_1}{v_0}(x-\omega^2y) \\
        &=\left(\frac{t_1}{t_0}+\omega^2\frac{u_1}{u_0}+\omega\frac{v_1}{v_0}\right)x
        -\left(\frac{t_1}{t_0}+\frac{u_1}{u_0}+\frac{v_1}{v_0}\right)y.
    \end{align*}
    This leads us to $[x:y:z:w]=[s_0AC:s_0BC:-s_1AC:-s_1AD]$ where
    \begin{align*}
        A&=\frac{t_1}{t_0}+\frac{u_1}{u_0}+\frac{v_1}{v_0}, \\
        B&=\frac{t_1}{t_0}+\omega^2\frac{u_1}{u_0}+\omega\frac{v_1}{v_0} \\
        C&=\frac{t_0}{t_1}+\frac{u_0}{u_1}+\frac{v_0}{v_1} \\
        D&=\frac{t_0}{t_1}+\omega\frac{u_0}{u_1}+\omega^2\frac{v_0}{v_1}.
    \end{align*}
    It follows that the rational map $\psi:S\dashrightarrow S_{Sch}$ defined by
    \[
        ([s_0:s_1],[t_0:t_1],[u_0:u_1],[v_0:v_1])\;\longmapsto\;[s_0AC:s_0BC:-s_1AC:-s_1AD]
    \]
    is inverse of $\varphi$, and therefore $\varphi$ and $\psi$ are birational. These birational maps extend to isomorphisms by minimality of a $K3$ surface.

\end{proof}

By the isomorphism $S_{T_{192}}\cong S_{Sch}$ constructed in the proof we deduce that $S_{Sch}$ also has a polarization-preserving automorphism group isomorphic to $T_{192}\rtimes C_6$ generated by
\begin{align*}
    \begin{pmatrix}
        \gamma&0\\
        0&E
    \end{pmatrix},
    \begin{pmatrix}
        \delta&0\\
        0&E
    \end{pmatrix},
    \begin{pmatrix}
        0&E\\
        E&0
    \end{pmatrix},
    \begin{pmatrix}
        \sqrt{-1}E&0\\
        0&E
    \end{pmatrix}\in\PGL_4(k),
\end{align*}
where
\begin{align*}
    \gamma =\begin{pmatrix}
        1&0 \\ 0&\omega
    \end{pmatrix},
    \;\delta=\sqrt{\frac{-1}{3}}
    \begin{pmatrix}
        1&-1\\-2&-1
    \end{pmatrix}.
\end{align*}

\section{Application of the isomorphism}

Let $h\in NS(X)$ be the class of a polarization of a $K3$ surface $X$. We denote by $\Aut(X,h)$ the subgroup of $\Aut(X)$ that preserves the class $h$ and by $\Aut_s(X,h)$ its symplectic subgroup.

Let $h_{24}$ be the class of the polarization of $S$ and let $h_4$ be the class of the polarization of $S_{Sch}$. In \cite{Brandhorst2021ExtensionsOM} the polarization-preserving automorphism groups are calculated, and we have 
\begin{align*}
    \Aut(S,h_{24})&\cong \mathfrak{S}_4\times\mathfrak{S}_4, \\
    \Aut_s(S,h_{24})&\cong \mathfrak{A}_{4,4}
\end{align*}
and
\begin{align*}
    \Aut(S,h_{4})&\cong T_{192}\rtimes C_6, \\
    \Aut_s(S,h_{4})&\cong T_{192}.
\end{align*}
over $\complex$.

The purpose of this section is to prove the following
\begin{theorem}\label{intersection}
    The intersection of $\Aut(S,h_{24})$ and $\Aut(S,h_4)$ in $\Aut(S)$ is isomorphic to $C_4\rtimes \mathfrak{S}_4$ and its symplectic subgroup is isomorphic to $C_2\times\mathfrak{S}_4$.
\end{theorem}

To begin with, we focus on a configuration of rational curves on $S$.

\begin{lemma}
    The projections $p_i:S\to\projsp^1$ to the $i$th factor are elliptic fibrations and they have exactly $6$ singular fibers of type \emph{IV} in Kodaira's notation.
\end{lemma}
\begin{proof}
    By symmetry, it is enough to show the statement for $p_1$. Set
    \[
        C_{12}=\{([0:1],[1:0],[a:b],[a:-b])\in S\mid [a:b]\in\projsp^1\}.
    \]
    Then $C_{12}$ is a smooth rational curve. Similarly, for $1\leq i,j\leq,4$ with $i\neq j$, let $C_{ij}$ denote the rational curve whose $i$th and $j$th coordinates are $0$ and $\infty$, respectively.

    $p_1$ has a singular fibers over $0$, namely $C_{12}+C_{13}+C_{14}$. This forms the configuration of type \emph{IV} since the three smooth rational curves meet at a single point $(0,\infty,\infty,\infty)$. The class of fiber of $p_1$ is primitive because $(p_1^*(0).C_{34})=1$. Hence $p_1$ is an elliptic fibration by Theorem~\ref{ellfiber}.

    By symmetry $p_1$ also has singular fibers of type \emph{IV} over $\pm 1$, $\pm\sqrt{-1}$ and $\infty$.
    Since a singular fiber of type \emph{IV} has Euler number $4$, the Euler number of these fibers adds up to $24$, which equals the Euler number of a $K3$ surface. Thus, there is no other singular fibers of $p_1$.
\end{proof}

The polarization class $h_{24}$ is represented by the divisor of the sum of fibers of each $p_i$, say $\sum_{i\neq j}C_{ij}$. 

Set $f_{mn}\in H^0(S_{Sch},\mathcal{O}(1))$ ($m=0,1$, $n=1,\ldots,4$) as follows:
\begin{align*}
    f_{01}&=x,\, f_{02}=x-y,\, f_{03}=x-\omega y,\, f_{04}=x-\omega^2y,\\
    f_{11}&=z,\, f_{12}=z-w,\, f_{13}=z-\omega^2w,\, f_{14}=z-\omega w.
\end{align*}
And let $\ell_{ij}\defeq(f_{0i}=f_{1j}=0)$ ($1\leq i,j\leq 4$). The dual graph of configuration of these lines is $K_{4,4}$ (cf. \cite{Inosekum}). The isomorphism $\psi:S\to S_{Sch}$ maps $C_{ij}$ to the lines $\ell_{ij}$. Let $C\defeq\sum_{i\neq j}\ell_{ij}$.

\begin{lemma}
    An automorphism $g\in\Aut(S,h_4)$ is contained in $\Aut(S,h_{24})$ if and only if it stabilizes the set $\{\ell_{ii}\mid 1\leq i\leq 4\}$. 
\end{lemma}
\begin{proof}
    $\Aut(S,h_4)$ acts on the set of lines $A\defeq\{\ell_{ij}\mid 1\leq i,j\leq 4 \}$ by permutation.
    
    Take an element $g\in\Aut(S,h_4)$. Then $g(C)$ is a sum of $12$ distinct lines in $A$. If, in addition, $g$ is contained in $\Aut(S,h_{24})$, then the intersection numbers with any divisor coincide for $C$ and $g(C)$ since the numerical classes of $C$ and $g(C)$ are the same. 
    
    Suppose that $g(C)$ contains $\ell_{11}$. Then $g(C)$ is the sum of $\ell_{11}$ and $11$ other lines in $A$. The line $\ell_{11}$ intersects transversally with $6$ lines in $A$, i.e. $\ell_{i1}$ ($i=2,3,4$) and $\ell_{1j}$ ($j=2,3,4$), and does not intersect with the rest. Therefore we obtain the inequality
    \[
        (\ell_{11}.g(C))\leq -2+6=4<(\ell_{11}.C), 
    \]
    which is absurd. Hence $g(C)$ does not contain $\ell_{11}$, nor does it contain $\ell_{ii}$ for $i=2,3,4$ by the same argument. It follows that any automorphism $g\in\Aut(S,h_4)\cap\Aut(S,h_{24})$ preserves the set $B\defeq\{\ell_{ii}\mid 1\leq i\leq 4\}$. 

    Conversely, if an element $g\in\Aut(S,h_4)$ maps $B$ into itself, then $g$ also preserves the complement $A\setminus B=\{\ell_{ij}\mid 1\leq i,j\leq,4, i\neq j\}$. Hence $g$ fix the polarization $h_{24}$ which is the class of $\sum_{i\neq j}\ell_{ij}$, and thus $g\in\Aut(S,h_{24})$.
\end{proof}

We are ready to prove Theorem~\ref{intersection}.
\begin{proof}[Proof of Theorem~\ref{intersection}]
    The action of $\Aut(S,h_4)$ on the set of hyperplanes
    \[
    \{(f_{mn}=0)\mid m=0,1, n=1,\ldots,4\}
    \]
    induces a homomorphism of groups $\Aut(S,h_4)\to\mathfrak{S}_8$. Here we identify the hyperplane $(f_{mn}=0)$ with the point $4m+n$. Under this identification, the generators of $\Aut(S,h_4)$ are mapped as follows:
    \begin{align*}
        \begin{pmatrix}
            \gamma&0\\
            0&E
        \end{pmatrix}&\mapsto (2,4,3),
        \begin{pmatrix}
            \delta&0\\
            0&E
        \end{pmatrix}\mapsto (1,2)(3,4), \\
        \begin{pmatrix}
            0&E\\
            E&0
        \end{pmatrix}&\mapsto (1,5)(2,6)(3,8)(4,7),
        \begin{pmatrix}
            \sqrt{-1}E&0\\
            0&E
        \end{pmatrix}\mapsto\id.
    \end{align*}
    
    The stabilizer of the set $B$ is the inverse image of the subgroup of $\mathfrak{S}_8$ that preserves the partition $\{1,5\}\sqcup \{2,6\}\sqcup\{3,7\}\sqcup\{4,8\}$. Therefore $\Aut(S,h_4)\cap\Aut(S,h_{24})$ is generated by
    \[
        \begin{pmatrix}
            \gamma&0\\
            0&\gamma^2
        \end{pmatrix},
        \begin{pmatrix}
            \delta&0\\
            0&\delta
        \end{pmatrix},
        \begin{pmatrix}
            0&E\\
            E&0
        \end{pmatrix},
        \begin{pmatrix}
            \sqrt{-1}E&0\\
            0&E
        \end{pmatrix}
    \]
    and is isomorphic to $C_4\rtimes S_4$, the group of GAP ID [96,187]. The symplectic subgroup $\Aut_s(S,h_4)\cap\Aut_s(S,h_{24})$ is generated by
    \[
        \begin{pmatrix}
            \gamma&0\\
            0&\gamma^2
        \end{pmatrix},
        \begin{pmatrix}
            \delta&0\\
            0&\delta
        \end{pmatrix},
        \begin{pmatrix}
            0&E\\
            \sqrt{-1}E&0
        \end{pmatrix}
    \]
    and is isomorphic to $C_2\times\mathfrak{S}_4$.
\end{proof}

\bibliography{A44arxiv.bib}
\bibliographystyle{plain}

\end{document}